\documentclass{article}

\usepackage{amsmath,amsfonts,amsthm,amssymb,amscd,color,xcolor,mathrsfs,verbatim,microtype}
\usepackage{graphicx,eurosym}
\usepackage{hyperref}
\usepackage{mathtools}

\usepackage[applemac]{inputenc}

\usepackage[cyr]{aeguill}

\colorlet{darkblue}{blue!50!black}

\hypersetup{
    colorlinks,%
    citecolor=blue,%
    filecolor=red,%
    linkcolor=darkblue,%
    urlcolor=blue,%
    pdfnewwindow=true,%
    pdfstartview={FitH}
}

\usepackage{graphicx,amscd,mathrsfs,wrapfig,mathrsfs,lipsum}
\usepackage{eufrak}
\usepackage{float}
\usepackage{tikz}
\usepackage{multicol}
\usepackage{caption}
\usetikzlibrary{arrows}
\usepackage{capt-of}

\colorlet{darkblue}{blue!50!black}

\renewcommand{\Re}{\mathop{\rm Re}\nolimits}

\newcommand{\p}{\partial}

\newcommand{\Q}{{\mathbb Q}}
\newcommand{\K}{{\mathbb K}}

\newcommand{\D}{{\mathbb D}}

\newcommand{\C}{{\mathbb C}}

\newcommand{\mek}{{\mathbf 1}}

\newcommand{\R}{{\mathbb R}}

\newcommand{\N}{{\mathbb N}}
\newcommand{\la}{\lambda}

\newcommand{\La}{\Lambda}
\newcommand{\ty}{\infty}

\newcommand{\de}{\delta}

\newcommand{\aA}{{\cal A}}

\newcommand{\FF}{{\cal F}}
\newcommand{\GG}{{\cal G}}
\newcommand{\HH}{{\cal H}}

\newcommand{\KK}{{\cal K}}

\newcommand{\RR}{{\cal R}}

\newcommand{\sS}{{\cal S}}
\newcommand{\TT}{{\cal T}}

\newcommand{\lag}{\langle}
\newcommand{\rag}{\rangle}

\newcommand{\dd}{{\textup d}}

\newcommand{\lspan}{\mathop{\rm span}\nolimits}

\theoremstyle{plain}
\newtheorem*{mt}{Main Theorem}

\newtheorem*{lemma*}{Lemma}
\newtheorem{theorem}{Theorem}[section]
\newtheorem{lemma}[theorem]{Lemma}
\newtheorem{proposition}[theorem]{Proposition}
\newtheorem{corollary}[theorem]{Corollary}
\theoremstyle{definition}
\newtheorem{definition}[theorem]{Definition}

\theoremstyle{remark}

\newtheorem{remark}[theorem]{Remark}

\numberwithin{equation}{section}

\begin{document}

\author{Alessandro Duca\,\footnote{Universit\'e Paris-Saclay, UVSQ, CNRS, Laboratoire de Math\'ematiques de Versailles, 78000, Versailles, France;  e-mail: \href{mailto:alessandro.duca@uvsq.fr}{Alessandro.Duca@uvsq.fr}} \and  Vahagn~Nersesyan\,\footnote{NYU-ECNU Institute of Mathematical Sciences, NYU Shanghai, 3663 Zhongshan Road North, Shanghai, 200062, China, e-mail: \href{mailto:vahagn.nersesyan@nyu.edu}{Vahagn.Nersesyan@nyu.edu}}\, \footnote{Universit\'e Paris-Saclay, UVSQ, CNRS, Laboratoire de Math\'ematiques de Versailles, 78000, Versailles, France}
}

 \date{\today}

\title{Local exact controllability    of the 1D nonlinear  Schr\"odinger~equation in the case of Dirichlet boundary conditions}
\date{\today}
\maketitle

\begin{abstract}
We consider the 1D nonlinear Schr\"odinger   equation with bilinear control.
In the case of  Neumann boundary conditions, 
  local exact controllability of this equation near the ground state
    has been proved by  Beauchard and Laurent~\cite{BL-2010}. 
In this paper, we study the case of Dirichlet boundary conditions.~To establish  the controllability of the linearised     equation, we~use a bilinear control   acting    through four directions:~three Fourier modes and one generic direction.~The~Fourier modes are appropriately chosen so that they satisfy a  saturation property.~These modes allow to control approximately  the linearised Schr\"odinger equation.~We show that the reachable  set for the linearised    equation   is closed.~This is achieved by representing the resolving operator        as a sum of two linear continuous  mappings: one   is surjective (here  the control in   generic  direction~is~used) and the other    is   compact.~A mapping with   dense and closed image is   surjective, so the linearised Schr\"odinger   equation is exactly controllable.    Then   local exact controllability of the nonlinear equation   is derived using the inverse mapping theorem.



\end{abstract}

\newpage

   \tableofcontents

\setcounter{section}{-1}

\section{Introduction}
\label{S:0}

We study the  controllability of the   one-dimensional   nonlinear 
Schr\"odinger (NLS) equation with bilinear control and   Dirichlet boundary conditions.~To simplify the presentation, 
  we    consider in this introduction the case of the cubic NLS equation   
      \begin{gather} 
 	i \p_t \psi  =-\p_{xx}^2 \psi +\kappa  |\psi|^2 \psi +\lag u(t),Q(x) \rag \psi, \quad x\in I=(0,1), \label{0.1}\\
 	 	\psi(t,0) =\psi(t,1)=0,\label{0.2}
 	 	   \end{gather} where   
  $Q:    I\to \R^q$ is a  given 
    external        field
     and $\kappa $ is a       real number.~We fix any~$T>0$  and  consider the amplitude $u:[0,T]\to \R^q$ as a  control term and   the solution at time~$T$, i.e.,~$\psi(T)$,  as a  state.

   To formulate the main result of this paper, let us introduce some notation. We   consider $L^2(I;\C)$
    as a real Hilbert space endowed  with the scalar product  
    $$
    \lag f,g \rag_{L^2}= \Re\int_0^1 f(x)\overline{g(x)} \dd x
    $$and the corresponding norm $ \|\cdot\|_{L^2}$.
 Let   $A$ be the Dirichlet Laplacian operator
     $$
     A=-\p_{xx}^2,\quad \quad \D(A)=H^2\cap H^1_0(I;\C),
     $$and  let
       $\phi_k(x)=\sqrt{2}\sin(k\pi x)$, $k\ge1$ be its eigenfunctions   associated with   the eigenvalues $\la_k=k^2\pi^2$.   
    We    use the spaces $H^s_{(0)}=\D(A^\frac{s}{2})$, $s\ge0$ endowed with the scalar products $ \lag f, g\rag_{(s)} = \lag A^\frac{s}{2} f, A^\frac{s}{2} g\rag_{L^2}$ and the corresponding norms~$\|\cdot\|_{(s)}$. 
    The system \eqref{0.1}, \eqref{0.2} is supplemented with the initial condition
  \begin{equation}\label{0.3}
   	 	 		\psi(0,x)=\psi_0(x),
   \end{equation} 
   which is assumed to belong to the  unit sphere $\sS$ in $L^2(I;\C)$. 
The following is the main result of this paper.
\begin{mt}  
Assume that $Q=(Q_1,\ldots, Q_q)$  is a smooth field such that the vector space 
\begin{equation}\label{0.4}
\Q=\lspan_\R\{Q_k:k=1,\ldots, q\}	
\end{equation}
 contains the functions $\mek$, $\cos(\pi x)$,   $\cos(2\pi x)$, and a function $\mu$
   verifying the inequality
\begin{equation}\label{0.5}
|\lag \mu \phi_1, \phi_k\rag_{L^2}|	\ge \frac{c}{k^3}, \quad k\ge1
\end{equation}for some number $c>0$.~Then 
 there is an at most countable set $\KK\subset (-\ty, 0)$ such that, for any $\kappa\in \R\setminus \KK$, the NLS equation   is locally exactly controllable  near the ground state $\phi_1$.~More precisely,   for any $T>0$, there is a number $\de>0$ such that, for  any $\psi_0, \psi_1 \in H^3_{(0)}(I;\C)\cap \sS$ with  
 $$
 \|\psi_j-\phi_1\|_{(3)}<\de, \quad j=0,1, 
 $$  
 there is a control $u\in L^2([0,T];\R^q)$ and a    solution $\psi\in C([0,T];H^3_{(0)}(I;\C))$   of the problem~\eqref{0.1}-\eqref{0.3}   satisfying  $\psi(T)= \psi_1$.
 \end{mt}
 A more general version of this theorem is stated  in Section \ref{S:2} (see Theorem~\ref{T:2.2}). In that version,  the   nonlinear term has  the form $|\psi|^{2p} \psi$ with any integer~$p\ge1$, and the conditions on the field $Q$ and the number $\kappa$ are formulated in terms of a general saturation property.

 The controllability of the Schr\"odinger     equation \eqref{0.1} has been extensively studied in the literature in the case $\kappa=0$.~Note that in that case, even if the equation is linear in $\psi$, the associated control problem is still nonlinear. Ball,~Marsden, and Slemrod~\cite{BMS-82} proved that the reachable set for this equation 
 from any initial~condition in~$H^2_{(0)}(I;\C)\cap \sS$   with controls in  $L^2$ has   empty interior in~$H^2_{(0)}(I;\C)\cap \sS$.~In particular, this means that the problem is not locally exactly controllable in that phase space.~Beauchard~\cite{B-2005} obtained the first positive controllability result:~in the case~$Q(x)=x$,  she    proved      local exact   controllability  in some~$H^7_{(0)}$-neighborhood of any eigenstate by using a Nash--Moser theorem.~Beauchard and Coron~\cite{BC-2006}   obtained exact controllability between   neighborhoods of different eigenstates.~Later, in the paper~\cite{BL-2010}, Beauchard and Laurent  found a way to use the classical inverse mapping theorem to prove   local exact controllability of the    Schr\"odinger equation; more precisely,    they proved exact controllability in some~$H^3_{(0)}$-neighborhood of any eigenstate in the case when~$Q=\mu$ satisfies condition~\eqref{0.5}.~The~methods of~\cite{BL-2010} have been further developed by  Morancey and the authors of this paper  in~\cite{Morgan-14, MN-15, AD-20} to study   simultaneous exact controllability of several  Schr\"odinger~equations.

 All the above papers deal with the one-dimensional Schr\"odinger     equation. In~the multidimensional case,     exact controllability  remains an   open problem. In~that situation,      approximate controllability property  has been studied by many authors;
for the first results we refer the reader  to the works by Boscain et al.~\cite{CMSB-2009, BCCS-12}, Mirrahimi~\cite{MM-09}, and the second author~\cite{VN-2010}.


In the case of the NLS equation \eqref{0.1} with $\kappa\neq 0$ and Neumann boundary conditions,   local exact controllability is established by Beauchard and Laurent~\cite{BL-2010}.~They use  the inverse mapping theorem and     exact controllability of the linearised equation.~The latter is  proved by using a convenient change of the unknown that reduces  the original problem  to another  linear system with explicit    spectrum. The controllability of the reduced system is proved by using a moment problem approach and the Ingham inequality.  In the case of Dirichlet boundary conditions, it is not clear whether such reduction is possible, so  we proceed in a different~way.

We  note that when  $\mek$ and $\cos(2\pi x)$ belong to the vector space~$\Q$ defined by~\eqref{0.4},     the ground state $\phi_1$ is a
stationary solution of the  NLS equation~\eqref{0.1}  corresponding to some constant control $u$.~The linearisation of the equation   around the couple $(\phi_1,u)$  is given by 
\begin{equation}\label{0.6}
 	i \p_t \xi=-\p_{xx}^2 \xi-\pi^2\xi +2 \kappa  \phi_1^2 \Re(\xi)   +   \lag v(t),Q(x) \rag \phi_1.
\end{equation}
We prove   exact controllability of this equation  in two steps.~First, 
 we  show that if the number $\kappa$ is in the complement of some at most countable set $\KK$, then the directions~$\mek$ and $\cos(\pi x)$ are saturating.~As a consequence, we obtain   approximate controllability of Eq.~\eqref{0.6}.
  The  saturation argument employed here is inspired by the 
 papers of Agrachev, Sarychev~\cite{AS-2005,AS-2006} and Shirikyan~\cite{shirikyan-cmp2006}, which study      approximate controllability  of the nonlinear Navier--Stokes and Euler systems.

 Next we show that   approximate controllability of Eq.~\eqref{0.6}  implies its exact controllability. To this end,  we decompose the solution   as follows $\xi=\xi_1+\xi_2$, where $\xi_1$ and~$\xi_2$ are    solutions of   equations
 \begin{align}  
 	i \p_t \xi_1&=-\p_{xx}^2 \xi_1-\pi^2\xi_1  +   \lag v(t),Q(x) \rag \phi_1, \label{0.7} \\
 	i \p_t \xi_2&=-\p_{xx}^2 \xi_2-\pi^2\xi_2 +2 \kappa  \phi_1^2 \Re(\xi_1+\xi_2). \label{0.8}  \end{align} 
 	Eq.~\eqref{0.7} is exactly controllable.~Indeed, as in \cite{BL-2010}, this can be seen  by   rewriting  the control system    as a moment
  problem and then by solving  it with the help of     the Ingham inequality and the assumption \eqref{0.5}.~On the other hand, we show that the resolving operator   of Eq.~\eqref{0.8} is compact.~According to a functional analysis result, in a Banach space, the sum of compact and    surjective linear continuous mappings  has   closed image. On the other hand, this image is dense, by    approximate controllability of Eq.~\eqref{0.6}.~An operator with    closed and dense image is obviously surjective, so the linearised Schr\"odinger equation~\eqref{0.6}    is exactly controllable.  Applying the inverse mapping theorem, we derive    local exact controllability of the nonlinear equation.

In the case $\kappa\in \KK$, the linearised   control system  may possibly miss    one direction. However, we expect that using the nonlinear term one can prove that the
result of   the Main Theorem   still remains true. It~would be natural to study this case by applying  a   power series expansion   in the spirit of the paper~\cite{CC-04} by Coron and Cr\'epeau  (see also~\cite{coron2007, BC-2006}). This~question  will be considered~elsewhere.

    The use of   saturation property to prove  approximate controllability of the linearised Schr\"odinger equation and the argument    allowing  to derive   exact controllability from    approximate controllability are among the novelties of this paper.    We believe these  arguments can be employed   in other useful situations.

Global   controllability of the NLS equation \eqref{0.1} with $\kappa\neq 0$ is  a  challenging open problem. First results in this direction have been obtained   recently  by the authors~\cite{DN-2021} and by Coron et al.~\cite{CXZ-21},  who consider    approximate controllability   between some particular states  (in a semiclassical sense in the second reference).~From the Main~Theorem and the time reversibility   of the Schr\"odinger equation it follows that   global exact controllability will be established if one shows   approximate controllability to  the ground state~$\phi_1$ in the $H^3_{(0)}$-norm (see~Theorem~3.2 in~\cite{VN-2010} for the case $\kappa=0$).

 \subsubsection*{Acknowledgement}  
    This paper was finalised when the second author was visiting the  School of Mathematical Sciences of Shanghai Jiao Tong University. He thanks the institute for  hospitality.

 \subsubsection*{Notation}

 In this paper, we use the following notation.
 
 \smallskip
\noindent
 $\lag \cdot,  \cdot\rag$ and $|\cdot|$ denote the Euclidian   scalar product and   norm in    $\R^q$.  
 
  \smallskip
\noindent 
$\mek$ is the function identically equal to $1$ on the interval $I=(0,1)$.

 \smallskip
\noindent
  $L^2=L^2(I;\C)$  and $H^s=H^s(I;\C)$, $s> 0$ are the usual  Lebesgue and Sobolev   spaces of functions $g:I\to \C$   with the    norms $\|\cdot\|_{L^2}$ and $\|\cdot\|_{H^s}$.  In the case of the spaces of real-valued functions, we   write $L^2(I;\R)$  and $H^s(I;\R)$.
  
 \smallskip
\noindent
 For any $V\in H^3(I;\R)$, we denote by $A_V$    the Schr\"odinger operator 
\begin{equation}\label{0.9}
     A_V=-\p_{xx}^2+V,\quad \quad \D(A_V)=H^2\cap H^1_0(I;\C).
\end{equation}
      $\{\phi_{k,V}\}$ is an orthonormal basis in $L^2(I;\R)$ formed by  eigenfunctions of      $A_V$, and~$ \la_{1,V}< \ldots< \la_{k,V}<\ldots $ are~the corresponding eigenvalues.   
  
   \smallskip
\noindent  $H^s_{(V)}=\D(A_V^\frac{s}{2})$, $s\ge0$. Note that
 $$ 
H^3_{(V)}=H^3_{(0)} = \left\{ \psi\in H^3(I,\C):  \psi|_{x=0,1}=\psi''|_{x=0,1}=0 \right\}.
$$
  
 \noindent 
$\ell^2=\left\{\{a_k\}_{k\ge1}\in  \C^\N: \sum_{k=1}^{+\ty}|a_k|^2<+\ty\right\}$,  $\ell^2_r=\left\{\{a_k\}_{k\ge1}\in \ell^2 :  a_1\in\R\right\}$.

 \smallskip
\noindent 
We   write $J_T$ instead of $[0,T]$.

 \smallskip
\noindent
 Let $X$ be a Banach space endowed with a norm $\|\cdot\|$.

 \smallskip
\noindent $B_X(a, r)$ denotes the closed ball in $X$ of radius $r > 0$ centred at $a\in X$.

  \smallskip
\noindent  $L^2(J_T;X)$   is      the
space of Borel-measurable functions $g:J_T\to X$ with the norm 
$$
\|g\|_{L^2(J_T;X)}=\left(\int_0^T\|g(t)\|_X^2\dd t
\right)^{\frac12}.
$$

 \smallskip
\noindent 
$C(J_T;X)$  is the space of continuous functions $g:J_T\to X$   with the norm 
$\|g\|_{C(J_T;X)}=\max_{t\in J_T}\|g(t)\|_X.$

\section{Controllability of the linearised equation}\label{S:1}

In this  section, we  study 
   the  controllability   of  the following linear Schr\"odinger   equation (cf. Eq.~\eqref{0.6}):  
\begin{gather}  
 	i \p_t \xi=-\p_{xx}^2 \xi+V(x)\xi-\la\xi + W(x) \Re(\xi)   +   \lag v(t),Q(x) \rag \phi(x), \label{1.1}\\
 	 	\xi(t,0)=\xi(t,1)=0, \label{1.2}
 	 	\end{gather} 
where   $Q:    I\to \R^q$ is a given field,   $V,W:    I\to \R$ are given potentials,~$\phi=\phi_{1,V}$   is the ground state of the Schr\"odinger operator $A_V$ (see \eqref{0.9}),  and $\la=\la_{1,V}$ is the associated eigenvalue.~The following  well-posedness result is a   consequence of Proposition~\ref{P:3.1}.
\begin{proposition}\label{P:1.1}
	For any   $T>0,  Q  \in H^3 (I,\R^q),  V,W\in H^3(I;\R),      v \in L^2(J_T;\R^q)$,   and $\xi_0\in H^3_{(0)},$    there is a unique   solution  $\xi \in C(J_T;H^3_{(0)} )$  of the   problem  \eqref{1.1},~\eqref{1.2}     satisfying    $\xi(0)=\xi_0$.~Moreover,  there is a constant $C_T>0$ such~that 
$$
\|\xi\|_{C(J_T;H^3_{(0)} )}\le C_T \left (\|\xi_0\|_{(3)}+\|v\|_{L^2(J_T;\R^q)}\right).
$$
\end{proposition}
Let 
$$
\TT_{\phi}=\{\psi\in L^2(I,\C): \Re \lag\psi,\phi   \rag_{L^2}=0\}
$$ be the tangent space to the unit sphere $\sS$ at $\phi$. As the functions  $Q,V,W, $ and $v$ are real-valued, we have    $\xi(t)\in \TT_{\phi}$
 for any~$t\in J_T$, provided that  $\xi_0\in \TT_{\phi}$.   Let
 $$
 \RR: H^3_{(0)}\times   L^2(J_T;\R^q)\to  C(J_T;H^3_{(0)} ), \quad 
 (\xi_0,v)\mapsto \xi
 $$ be the resolving operator of the problem \eqref{1.1},~\eqref{1.2}, and let $\RR_T$ be its restriction at time~$T$.  
  In     Section \ref{S:1.1},  we show that  this problem is approximately controllable under some saturation condition. Then, in Section \ref{S:1.2}, assuming additionally that the vector space $\Q$  spanned by the components of~$Q$ contains a function~$\mu$   satisfying an inequality similar to~\eqref{0.5},  we~prove   exact controllability of the problem.

\subsection{Approximate controllability}\label{S:1.1}

Assume that the field $Q$ and the potentials $V,W$ are smooth.~Let   us define finite-dimensional vector spaces by
$$
\HH=\lspan_\R\{Q_j \phi:\,\, j=1,\ldots,q\}
$$
and  
\begin{equation}\label{1.3}
\FF(\HH)=\lspan_\R\left\{f+i(-\p_{xx}^2 g + Vg-\la g+ W \Re(g))  :\,\,  f,g\in \HH\right\}.
\end{equation} In the spirit of the papers~\cite{AS-2005,AS-2006,shirikyan-cmp2006},
we define a   non-decreasing  sequence of finite-dimensional spaces    $\{\HH_j\}$ in the following way 
  \begin{equation}\label{1.4}
\HH_0=\HH,\quad \HH_j=\FF(\HH_{j-1})\quad \textrm{for} \,\,\, j \geq
1,\quad \HH_\infty=\bigcup_{j=0}^\infty \HH_j. 
\end{equation} 
Let   ${\mathsf P}_1$   be   the orthogonal projection onto the closed subspace $H^3_{(0)}\cap \TT_{\phi}$ in~$H^3_{(0)}$.
\begin{definition}\label{D:1.2}
   We    say that a field $Q$ is  saturating  for the   problem  \eqref{1.1}, \eqref{1.2}        if~$\HH_\ty \subset H^3_{(0)}$ and  the projection   ${\mathsf P}_1\HH_\ty$  is dense in $ H^3_{(0)}\cap \TT_{\phi}$. 
\end{definition}
 \begin{proposition}\label{P:1.3}
 Assume that $Q$ is saturating.~Then
 the   problem  \eqref{1.1}, \eqref{1.2}       is approximately controllable in the sense that  the image of the linear mapping 
 $$
 \RR_T(0,\cdot):L^2(J_T;\R^q)\to   H^3_{(0)}\cap \TT_{\phi}, \quad v\mapsto \xi(T)
 $$   is~dense in~$ H^3_{(0)}\cap \TT_{\phi}$ for $T>0$.\end{proposition}
\begin{proof} 	{\it Step~1.~Reduction.} For any  $0\le \tau\le t\le T$,
let
 $$R(t,\tau):H^3_{(0)}  \to H^3_{(0)},\quad  \xi_0\mapsto \xi(T)
 $$  be   the  resolving operator of the   problem 
  \begin{gather*}
 	i \p_t \xi=-\p_{xx}^2 \xi +V(x)\xi-\la\xi+ W(x) \Re(\xi),\\
 	 	\xi(t,0)=\xi(t,1)=0,\\
 	 	\xi(\tau,x)=\xi_0. 
 	 	\end{gather*}
  Let the operator $ \aA:L^2(J_T; H^3_{(0)}) \to H^3_{(0)}$ be defined by
  	$$
	 \aA(v)= \int_0^T R(T,\tau) v(\tau)\, \dd \tau, \quad v \in L^2(J_T; H^3_{(0)}),
	$$    
	    and let     ${\mathsf P}_\HH$   be   the orthogonal projection onto~$\HH$ in~$H^3_{(0)}$.~The proposition will be proved if we show that 
	  the image of the   operator 
	$$
	\aA_1:L^2(J_T, H_{(0)}^3)\to H^3_{(0)}\cap \TT_{\phi}, \quad   \aA_1=\aA{\mathsf P}_\HH
	$$is dense in $H^3_{(0)}\cap \TT_{\phi}$. 
	The latter will be achieved by showing that
	  the kernel of the adjoint $\aA_1^*$ of $\aA_1$  is trivial. Note that   $\aA_1^*$  is  given by 
	$$
	\aA_1^*: H^3_{(0)}\cap \TT_{\phi} \to L^2(J_T, \HH), \quad z\mapsto {\mathsf P}_\HH  R(T,\cdot)^* z,
	$$ 
	 where $  R(T,\tau)^*:H^3_{(0)}\to H^3_{(0)}$ is the $H^3_{(0)}$-adjoint of $R(T,\tau)$, $\tau \in J_T.$  
	
	\smallskip
	{\it Step~2.~Triviality of the kernel of~$ \aA_1^*$.}~Let   $z$ be an arbitrary element of the kernel of~$ \aA_1^*$.~Our goal is to show that~$z=0$. To this end, we take any~$g\in\HH$ and note that   
$$
	\lag g,R (T,\tau)^*z\rag_{(3)}=0\quad \text{for almost any $\tau\in J_T$}.
$$ By continuity in $\tau$  of $R (T,\tau) g$, this is equivalent to  
\begin{equation}\label{1.5}
	\lag R (T,\tau) g, z\rag_{(3)}=0\quad \text{for any  $\tau\in J_T$}.
\end{equation}Taking $\tau=T$ in this equality, we see that  $z$ is orthogonal to $\HH$ in $H^3_{(0)}$. In what follows,  we  show that $z$ is orthogonal to   $\HH_j$    for any~$j\ge1$. This, together with the saturation  assumption, will imply that $z=0.$

 Let us fix any $T_1\in (0,T)$ and 
  rewrite \eqref{1.5} as follows:
\begin{equation} \label{1.6}
	\lag R (T_1,\tau) g, R(T,T_1)^*z\rag_{(3)}=0\quad \text{for any  $\tau\in J_{T_1}$}.
\end{equation}  Note that $\zeta(\tau)=R (T_1,\tau) g$  is the solution of the problem
   \begin{gather*} 
 	i \p_\tau \zeta=-\p_{xx}^2 \zeta +V(x)\zeta-\la\zeta+ W(x)\Re(\zeta),\\
 	 	\zeta(\tau,0)=\zeta(\tau,1)=0,\\
 	 	\zeta(T_1,x)=g(x). 
 	 	\end{gather*} 
Taking the derivative  of \eqref{1.6} in $\tau$ and choosing $\tau=T_1$, we get   
$$
	\lag i(-\p_{xx}^2 g +Vg-\la g+ W \Re(g)), R (T,T_1)^*z\rag_{(3)}=0\quad \text{for any  $g\in \HH$}.
$$Thus
$$
	\lag g, R(T,T_1)^*z\rag_{(3)} =0\quad \text{for any  $g\in \HH_1$}.
$$
As  $T_1\in (0,T)$ is arbitrary,   we see that $z$ is orthogonal to   $\HH_1$  in $H^3_{(0)}$. Iterating this argument, we derive   orthogonality of $z$ to   $\HH_j$ for any $j\ge1$. Since  ${\mathsf P}_1\HH_\ty$  is dense in $ H^3_{(0)}\cap \TT_{\phi}$, we conclude that $z=0$.    \end{proof}
 Let us close this section with   an example of saturating field $Q$. This example will be used in the proof of the Main Theorem formulated in the Introduction.  
 Let us   introduce the operator
\begin{equation}\label{1.7}
A_\kappa g    =-\p_{xx}^2 g -\pi^2g+2 \kappa  \phi_1^2  g, \quad\quad \D(A_\kappa)=H^2\cap H^1_0(I;\R),
\end{equation}and
let $\la_{1,\kappa}< \ldots<\la_{k,\kappa}<\ldots $ be the sequence of its eigenvalues.~The following lemma is proved in Section~\ref{S:3.3}. 
\begin{lemma}\label{L:1.4}
	There is an at most countable set $\K\subset (-\ty,0]$ such that, for any~$\kappa\in \R\setminus \K$ and any $k\ge1$, we~have~$\la_{k,\kappa}\neq 0$. 
\end{lemma}Recall that $\Q$ denotes the vector space spanned by the components of $Q$ (see~\eqref{0.4}).
The following proposition is proved in Section~\ref{S:3.4}.
\begin{proposition}\label{P:1.5}
Let   $V(x) = 0$ and   $W(x) =  2 \kappa  \phi_1^2(x)$ for any $x\in I$, let $\kappa\in \R\setminus \K$, where $\K$ is the set in Lemma~\ref{L:1.4},  and
let  $Q$ be a smooth field such that~$
\mek$, $\cos(\pi x) \in \Q$.~Then  $Q$ is saturating in the sense of Definition~\ref{D:1.2}. \end{proposition}

\subsection{Exact controllability}\label{S:1.2}

 The aim of this section is to prove the following proposition. 
\begin{proposition}\label{P:1.6}
Let $V$ and $W$ be smooth functions such that the following boundary conditions are verified:
\begin{equation}\label{1.8}
 	W(0)=W(1)=W'(0)=W'(1)=0.
 \end{equation} 
Moreover,    assume that $Q$ is saturating, and   there is a function $\mu \in \Q$ satisfying the inequality
\begin{equation}\label{1.9}
|\lag \mu \phi, \phi_{k,V}\rag_{L^2}|	\ge \frac{c}{k^3}, \quad k\ge1
\end{equation}for some $c>0.$
 Then the   problem  \eqref{1.1}, \eqref{1.2} is exactly controllable in the sense that  the   mapping $\RR_T(0,\cdot):L^2(J_T;\R^q)\to   H^3_{(0)}\cap \TT_{\phi}$  is surjective for any $T>0$. 
\end{proposition}
\begin{proof}
 To prove this proposition, we represent the solution $\xi$ of the   problem  \eqref{1.1},~\eqref{1.2}  as follows $\xi=\xi_1+\xi_2$, where $\xi_1$ and $\xi_2$ are the solutions of~problems
\begin{gather*}  
 	i \p_t \xi_1=-\p_{xx}^2 \xi_1+V(x)\xi_1-\la \xi_1  +   \lag v(t),Q(x) \rag \phi, \label{BBS1} \\
 	 	\xi_1(t,0)=\xi_1(t,1)=0, \label{BBS2} \\
 	 	\xi_1(0,x)=0\label{BBS3} 
 	 	\end{gather*} 
and 
\begin{gather*}  
 	i \p_t \xi_2=-\p_{xx}^2 \xi_2+ V(x)\xi_2-\la \xi_2 +W(x) \Re(\xi_1+\xi_2), \label{AAS1} \\
 	 	\xi_2(t,0)=\xi_2(t,1)=0\label{AAS2}   \\
 	 	\xi_2(0,x)=0.\label{AAS3} 
 	 	\end{gather*} 
 	  	Let 
 	  	$$
 	  	\RR^j_T (0,\cdot):    L^2(J_T;\R^q)\to  H^3_{(0)}\cap \TT_{\phi} , \quad v\mapsto \xi_j(T), \quad j=1,2
 	  	$$   be the  resolving operators of these problems.       	\begin{lemma}\label{L:1.7}
 	 		The mapping   $\RR_T^1(0,\cdot):L^2(J_T;\R^q)\to   H^3_{(0)}\cap \TT_{\phi}$  is surjective for any $T>0$. 
 	 	\end{lemma}
 	 	
 	 	\begin{lemma}\label{L:1.8}
 	 		The mapping   $\RR_T^2(0,\cdot):L^2(J_T;\R^q)\to   H^3_{(0)}\cap \TT_{\phi}$  is compact for any~$T>0$. 
 	 	\end{lemma} 	 	
 	 	Taking these lemmas for granted, let us complete the proof of the proposition. We have 
$$
\RR_T(0,\cdot)=\RR_T^1(0,\cdot)+\RR_T^2(0,\cdot),
$$where the linear bounded operators $\RR_T^1(0,\cdot)$ and $\RR_T^2(0,\cdot)$ are, respectively,   surjective     and    compact   from $ L^2(J_T;\R^q)$ to~$ H^3_{(0)}\cap \TT_{\phi}$. 
 Lemma~\ref{L:3.3}
   implies that the image of the mapping $\RR_T(0,\cdot)$ is closed in   $H^3_{(0)}\cap \TT_{\phi}$. On the other hand, by Proposition \ref{P:1.3}, the image of $\RR_T(0,\cdot)$ is dense in $H^3_{(0)}\cap \TT_{\phi}$. We conclude that~$\RR_T(0,\cdot)$ is surjective.
 	 	\end{proof}

 \begin{proof}[Proof of Lemma \ref{L:1.7}]
 Let us consider the problem
   \begin{gather*}  
 	i \p_t \xi_3=-\p_{xx}^2 \xi_3+V (x)\xi_3-\la \xi_3  +     v(t) \mu(x)   \phi,  \\
 	 	\xi_3(t,0)=\xi_3(t,1)=0,  \\
 	 	\xi_3(0,x)=0,
 	 	\end{gather*}
 	 	and denote by $\RR^3_T (0,\cdot):    L^2(J_T;\R)\to  H^3_{(0)}\cap \TT_{\phi}$ its resolving operator.~From the assumption that~$\mu\in\Q$  it follows that  the image of~$\RR^3_T (0,\cdot)$ is contained in that of~$\RR^1_T (0,\cdot)$, so it suffices  to prove the surjectivity of $\RR^3_T (0,\cdot)$. The latter is proved  by rewriting the system  as a moment problem which is then solved using the  	 	Ingham inequality (see Proposition~4 in~\cite{BL-2010}). Indeed, $\xi_3$ satisfies the equality 
 	 	$$
 	 	  \xi_3(t)= -i\int_0^te^{-i(A_V-\la)(t-s)} \left( v(s) \mu(x)   \phi \right) \dd s, \quad t\in J_T.
	$$
	We   write~$\xi_3(T)$   in the form
$$
\xi_3(T)=-i\sum_{k=1}^{+\ty} e^{-i(\la_{k,V}-\la)T} \lag \mu \phi,\phi_{k,V}\rag_{L^2} \phi_{k,V}\int_0^T e^{i (\la_{k,V}-\la) s} v(s) \dd s,
$$
which is equivalent to 
$$
\int_0^T e^{i (\la_{k,V}-\la) s} v(s) \dd s = \frac{i e^{i(\la_{k,V}-\la) T}}{\lag \mu \phi,\phi_{k,V}\rag_{L^2}} \int_0^1 \xi_3(T,x) \phi_{k,V}(x)\dd x, \quad k\ge1.
$$
In view of assumption \eqref{1.9},
for any $\tilde \xi  \in  H^3_{(0)}\cap \TT_{\phi}$, we have
$$
\left\{\frac{i e^{i(\la_{k,V}-\la) T}}{\lag \mu \phi,\phi_{k,V}\rag_{L^2}} \int_0^1 \tilde \xi(x) \phi_{k,V}(x)\dd x\right\}_{k\ge1}\in \ell^2_r.
$$
 By the asymptotic formula for the eigenvalues (e.g., see Theorem~4 in~\cite{PT-87}),     
\begin{equation}\label{1.10}
\lambda_{k,V} =k^2\pi^2+\int_0^1 V(x)\dd x + r_k, \quad \text{where  } \{r_k\}_{k\ge1}\in \ell^2.
\end{equation}Hence, 
  $\la_{k,V}-\la_{k-1,V}\to +\ty$ as $k\to +\ty$.~Applying Corollary~1 in~\cite{BL-2010}, we~find that  there is~$v\in L^2(J_T;\R)$ such that 
$$
\int_0^T e^{i\la_{k,V} s} v(s) \dd s = \frac{  ie^{i\la_{k,V} T}}{\lag \mu \phi,\phi_{k,V}\rag_{L^2}}  \int_0^1 \tilde \xi(x) \phi_{k,V}(x)\dd x, \quad k\ge1.
$$This shows that $\RR^3_T (0,\cdot)$ is surjective and completes the proof of Lemma~\ref{L:1.7}. \end{proof}

  \begin{proof}[Proof of Lemma \ref{L:1.8}]
  The proof    follows immediately from Corollary \ref{C:3.3}. Indeed, the operator $\RR_T^2(0,\cdot) $ is compact   since it can be represented as a    composition of  the compact operator~$ \Phi $ in Corollary~\ref{C:3.3} 
   with linear continuous mappings.
  \end{proof}

\section{Proof of the Main Theorem}\label{S:2}

Let us consider the NLS equation
      \begin{gather} 
 	i \p_t \psi  =-\p_{xx}^2 \psi+V(x)\psi +\kappa  |\psi|^{2p} \psi +\lag u(t),Q(x) \rag \psi, \quad x\in I, \label{2.1}\\
 	 	\psi(t,0) =\psi(t,1)=0,\label{2.2}
 	 	   \end{gather} 
 where    the field $Q$ and the potential $V$ are smooth, $p\ge1$ is an integer,
     and $\kappa $ is a       real number. 
The following proposition establishes      local   well-posedness of this equation.  It is proved in Section~\ref{S:3.2}.
 \begin{proposition} \label{P:2.1}
Assume that,   for some $T>0$, $\hat \psi_0\in H^3_{(0)}$,  and $\hat u \in L^2(J_T;\R^q)$,    there is a    solution $\hat \psi\in C(J_T;H^3_{(0)})$ of the   problem  \eqref{2.1}, \eqref{2.2} satisfying the~initial condition $\hat \psi(0)=\psi_0$.~Then there are positive  numbers
 $\delta=\delta(T,\Lambda)$  and~$C=C(T,\Lambda)$, where   
 \begin{equation}\label{2.3}
\Lambda=  \|\hat \psi\|_{C( J_T;H^3_{(0)} )}+\|\hat u\|_{L^2(J_T; \R^q)},
\end{equation}
such that the following properties hold.
\begin{enumerate}
\item[(i)] For any $\psi_0\in H^3_{(0)}$ and $u\in L^2(J_T;\R^q)$  satisfying  
\begin{equation}\label{2.4}
\|\psi_0-\hat \psi_0\|_{(3)}+ \|u-\hat u\|_{L^2(J_T;\R^q)}  <
\delta,
\end{equation}
  the problem  \eqref{2.1}, \eqref{2.2} has a unique
solution $ \psi \in C(J_T;H^3_{(0)})$ satisfying the initial condition $\psi(0)=\psi_0$.
\item[(ii)] 
Let $\Psi_T$ be the mapping  taking a couple  $(\psi_0,u)$ satisfying \eqref{2.4} to~$\psi(T)$. Then~$\Psi_T$   is $C^1$, and for any $v\in L^2(J_T;\R^q)$, we have $\partial_{u} \Psi_T(\hat\psi_0,\hat u)v=\xi(T) $, where $\xi$ is the  solution of   linearised system
       \begin{gather}
 	i \p_t \xi=-\p_{xx}^2 \xi+V(x)\psi +(p+1)\kappa  |\psi|^{2p} \xi +p \kappa \psi^2 |\psi|^{2(p-1)}\overline\xi\nonumber\\ + \lag \hat u(t),Q(x) \rag \xi+ \lag v(t),Q(x) \rag \hat\psi,\label{2.5}\\
 	 	\xi(t,0)=\xi(t,1)=0,\label{2.6}\\
 	 	\xi(0,x)=0. \label{2.7}
 	 	\end{gather} 
\end{enumerate}
\end{proposition}
The following theorem is a generalisation of the Main Theorem.
\begin{theorem}\label{T:2.2}
Assume that the field $Q$ is saturating in the sense of Definition~\ref{D:1.2}  with potentials $V(x)$ and~$W(x)=2p\kappa    \phi^{2p}(x)$, and  the vector space~$\Q$ contains the functions  $\mek, \phi^{2p}$,  and a function~$\mu$
 verifying     inequality \eqref{1.9}  for some   $c>0$.~Then, for any $T>0$,    there is a number $\de>0$ such that, for~any~$\psi_0, \psi_1 \in H^3_{(0)}\cap \sS $ with 
\begin{equation}\label{EERW}
 \|\psi_j-\phi_1\|_{(3)}<\de, \quad j=0,1, 
\end{equation}
 there is a control $u\in L^2([0,T];\R^q)$  and a   solution $\psi\in C([0,T];H^3_{(0)})$  of   the problem~\eqref{2.1}, \eqref{2.2}    satisfying    $\psi(0)=\psi_0$ and~$\psi(T)= \psi_1$.\end{theorem}
 \begin{proof}
 	From the assumption that the functions $\mek$ and $\phi^{2p}$  belong to $\Q$ it follows that $\hat \psi(t)=\phi$ is a stationary  solution of the problem \eqref{2.1}, \eqref{2.2} corresponding to some constant control  $\hat u(t) $ such that 
 	\begin{equation}\label{2.8}
 		\lag \hat u(t),Q(x) \rag=-\kappa  \phi(x)^{2p}-\la_{1,V}.
 	\end{equation}
   Proposition \ref{P:2.1} implies that the operator $\Psi_T:(\psi_0,u) \mapsto \psi(T)$  is well-defined and $C^1$-smooth in some neighbourhood of $(\phi, \hat u)$.~Furthermore, for     any $v\in L^2(J_T;\R^q)$, we have
    $\p_u\Psi_T(\phi,\hat u)v=\xi(T)$, where, in view of \eqref{2.8}, $\xi$ is the  solution of   equation
\begin{equation}\label{2.9}
 	i \p_t \xi=-\p_{xx}^2 \xi+V(x)\xi-\la_{1,V}\xi +2p\kappa    \phi^{2p}\Re(\xi)   + \lag v(t),Q(x) \rag \phi.
\end{equation}The conditions of Proposition~\ref{P:1.6} are satisfied for the  linearised system \eqref{2.9}, \eqref{2.6}, \eqref{2.7}, so it is exactly controllable.~Applying the  inverse mapping theorem to the mapping $u\mapsto \Psi_T(\phi,u)$, we find that  there is a number $\de>0$ such~that, for  any~$ \psi_0, \psi_1 \in H^3_{(0)}\cap \sS $ verifying  \eqref{EERW},  
 there are   controls $u_0,u_1\in L^2([0,T];\R^q)$ such that 
 $\Psi_T(\phi_1,u_0)=\overline{\psi_0}$ and~$\Psi_T(\phi_1,u)=\psi_1$.~By the time-reversibility pro\-perty of the Schr\"odinger   equation, $\Psi_T(\psi_0, w)=\phi$ with~$w(t)=u_0(T-t)$. Setting~$u(t)=w(t)$ for $t\in [0,T]$ and $u(t)=u_1(t-T)$ for $t\in [T,2T]$, we derive~$\Psi_{2T}(\psi_0,u)=\psi_1$.~As the time $T>0$ is arbitrary, we complete the proof of Theorem~\ref{T:2.2}.
  \end{proof}

\begin{proof}[Proof of the Main Theorem.]   Main Theorem is proved by  applying Theorem \ref{T:2.2} with $p=1$ and $V=0$. Indeed, the assumption that
 $\Q$ contains the functions~$\mek$ and   $\cos(2\pi x)$ implies that $\phi_1^2$ is also in $\Q$. Proposition \ref{P:1.5}
 implies the saturation assumption with      $W(x) =  2 \kappa  \phi_1^2(x)$ and  $\kappa\in \R \setminus \K$. 
 
 When\footnote{Note that zero belongs to the set  $\K$ in Lemma \ref{L:1.4}, since $\la_{1,0}=0$.} $\kappa=0$, the linearised problem is still exactly controllable by Lemma~\ref{L:1.7}, so the conclusions of the theorem hold in that case too. Thus we proved the Main Theorem with   $\KK=\K\setminus\{0\}\subset (-\ty,0)$.
 \end{proof}
 \begin{remark}\label{R:2.3}
	Let us emphasise that, under the conditions of the Main Theorem, the  linearised    problem  \eqref{1.1}, \eqref{1.2}  with  $V(x) = 0$,   $W(x) =  2 \kappa  \phi_1^2(x)$,       and $\kappa\in   \KK$ may  still be exactly controllable.~Indeed,  in that situation we only know that the linearised system can   miss at most one direction (see Corollary~\ref{C:3.4}). As~mentioned in the Introduction, even if the linearised  system misses one direction, the nonlinear system can still be   locally exactly controllable near the ground state.
	\end{remark}
\begin{remark}\label{R:2.4}
It is not difficult to construct examples of saturating fields in the case of any integer~$p\ge1$. Indeed, 
the results of Lemma~\ref{L:1.4} and Proposition~\ref{P:1.5}   generalise without difficulties     to the case when $V(x)=0$, $W(x)=2p\kappa    \phi^{2p}$, and~$\mek, \cos (\pi x), \ldots, \cos (N_p\pi x)\in\Q$ with sufficiently large integer $N_p\ge1$.

\end{remark}

  \section{Appendix}

\subsection{Well-posedness   of the  linear Schr\"odinger   equation}\label{S:3.1} 
Let us consider 
  the following linear Schr\"odinger   equation with a source term:  
\begin{gather}  
 	i \p_t \xi=-\p_{xx}^2 \xi+V(x)\xi + W(x) \Re(\xi)   +   f(t,x), \label{3.1}\\
 	 	\xi(t,0)=\xi(t,1)=0. \label{3.2} 
 	 	\end{gather} 
 	 	Well-posedness of this equation is essentially established in \cite{BL-2010}.
 \begin{proposition}\label{P:3.1}
	For any   $T>0$,  $V,W\in H^3(I;\R)$,       $f  \in L^2 (J_T; H^3\cap H^1_0(I;\C))  $, and $\xi_0\in H^3_{(0)},$  there is a unique   solution  $\xi \in C(J_T;H^3_{(0)} )$  of this   problem   with the initial condition     $\xi(0)=\xi_0$ in the sense that  
	$$
	  \xi(t)=e^{-iA_V t} \xi_0-i\int_0^te^{-iA_V(t-s)} \left( W \Re(\xi(s))   + f(s)\right) \dd s, \quad t\in J_T.
$$   Moreover,  there is a constant~$C_T>0$ such that 
$$
\|\xi\|_{C(J_T;H^3_{(0)} )}\le C_T \left (\|\xi_0\|_{(3)}+\|f\|_{L^2(J_T;H^3\cap H^1_0(I;\C))}\right).
$$
\end{proposition}
 By Lemma~1 in \cite{BL-2010}, the function
\begin{equation}\label{3.3}
G_f:t\mapsto \int_0^te^{-iA_V(t-s)} f(s)  \dd s \quad \textup{belongs to $C(J_T;H^3_{(0)} )$}
\end{equation}  	 
 and satisfies the  inequality\,\footnote{More precisely,   in the paper \cite{BL-2010}, the case $V=0$ is considered. The general case is proved in a similar way by using the asymptotics \eqref{1.10} for the eigenvalues of $A_V$.}  
\begin{equation}\label{3.4}
\|G_f\|_{C(J_T;H^3_{(0)} )}\le C_T \|f\|_{L^2(J_T;H^3\cap H^1_0(I;\C))}.
\end{equation}  Proposition \ref{P:3.1} is derived from this  by applying  a usual fixed point  argument	to the mapping $L: C(J_T;H^3_{(0)} ) \to C(J_T;H^3_{(0)} )$ defined by
\begin{align*}
	L(\xi)(t)=e^{-iA_V t} \xi_0-i\int_0^te^{-iA_V(t-s)} \left( W \Re(\xi(s))   + f(s)\right) &\dd s,  \\  \textup{for } \xi \in &C(J_T;H^3_{(0)}), \,\,\,   t\in J_T.
\end{align*} 
We shall not dwell on the details of the proof.

 The next lemma  follows from  the Arzel\`a--Ascoli theorem in a standard way; again we skip the details.
\begin{lemma}\label{L:3.2}Under the conditions of Proposition \ref{P:3.1},
	the mapping 
	$$
	G_\cdot:  L^2(J_T;H^3_{(0)})\to  C(J_T;H^3_{(0)} ),\quad  f\mapsto G_f
	$$ is compact.
\end{lemma}
Let us consider a particular case of the problem \eqref{3.1}, \eqref{3.2} given by 
 \begin{gather}  
 	i \p_t \xi=-\p_{xx}^2 \xi+ V(x)\xi +W(x) \Re(\xi+\eta), \label{3.5} \\
 	 	\xi(t,0)=\xi(t,1)=0,\label{3.6}   \\
 	 	\xi(0,x)=0,\label{3.7} 
 	 	\end{gather} and assume that $W$ satisfies the boundary conditions  \eqref{1.8}.
 	 	 Then
  the mapping $\eta\mapsto W(x) \Re(\eta)$ is continuous from $L^2(J_T;H_{(0)}^3)$ to itself, and the following is an immediate consequence of Lemma \ref{L:3.2}.
 \begin{corollary}\label{C:3.3}Under the conditions of Proposition \ref{P:3.1},
  the mapping  
$$
\Phi:L^2(J_T;H_{(0)}^3)\to C(J_T;H^3_{(0)} ), \quad \eta\mapsto \xi
$$is compact, where $\xi$ is the solution of the problem \eqref{3.5}-\eqref{3.7}.
 	\end{corollary}

\subsection{Local well-posedness of the  NLS   equation}\label{S:3.2} 

Here we prove Proposition \ref{P:2.1}.

 {\it Step~1.~Uniqueness.}~Let $\psi_1,\psi_2 \in C(J_T;H^3_{(0)} )$ be    solutions of~\eqref{2.1}, \eqref{2.2} with the same control $u\in L^2(J_T;\R^q)$ and the same initial condition~$\psi_1(0)=\psi_2(0)$. Then the difference  $\varphi= \psi_1-\psi_2$ satisfies  
    \begin{gather*} 
 	i \p_t \varphi  =-\p_{xx}^2 \varphi+V(x)\varphi +\kappa  |\psi_1|^{2p} \psi_1 - \kappa  |\psi_2|^{2p} \psi_2   +\lag u(t),Q(x) \rag \varphi,    \\
 	 	\varphi(t,0) =\varphi(t,1)=0, \\ 	 	 
 	 		 \varphi(0)=0.  
 	 	   \end{gather*} 
 From \eqref{3.4} it follows that
\begin{align*}
 \|\varphi(t)\|_{(3)}^2& \le C_T\Big( \||\psi_1|^{2p} \psi_1 -  |\psi_2|^{2p} \psi_2\|^2_{L^2(J_t;H^3\cap H^1_0(I;\C))}\\&\quad\quad\quad\quad\quad\quad+\|\lag u,Q \rag \varphi\|^2_{L^2(J_t;H^3\cap H^1_0(I;\C))}\Big)\\&
 \le  C_1\int_0^t  (1+|u(s)|^2 )\|  \varphi(s) \|_{(3)}^2\dd s , \quad t\in J_T,
\end{align*}where $C_1>0$ is a constant depending on $\|\psi_j\|_{C(J_T;H^3_{(0)})}$, $j = 1,2$. 
 Applying    the Gronwall inequality, we infer that $\varphi(t)=0$ for any $t\in J_T$.

 {\it Step~2.~Local-in-time existence.}~Let us take any $\psi_0\in H^3_{(0)}$ and $u\in L^2(J_T,\R^q)$ satisfying~\eqref{2.4},  any $T_1\in J_T$,      any  $\psi \in  B_{C(J_{T_1};H^3_{(0)} )} (\hat \psi,1)$,  and define    
 $$
	M(\psi)(t)=e^{-iA_V t} \psi_0-i\int_0^te^{-iA_V(t-s)} \left(\kappa  |\psi(s)|^{2p} \psi (s)+\lag u(s),Q(x) \rag \psi(s)\right) \dd s
$$ for $ t\in J_{T_1} $.~Then  \eqref{3.3} implies that   $M(\psi)\in C(J_{T_1};H^3_{(0)} )$.~Moreover, using~\eqref{3.4}, we get
\begin{align*}
\|M(\psi)(t) - \hat \psi(t)\|_{(3)}\le &C_2\Big(\|\psi_0-\hat \psi_0\|_{(3)}+ 	 \|u - \hat u\|_{L^2(J_T;\R^q)}\nonumber \\&+  	( \|\hat u\|_{L^2(J_{T_1};\R^q)} +T_1^{\frac12}) \|\psi - \hat \psi \|_{C(J_t;H^3_{(0)} ) } \Big), \quad t\in J_{T_1},  
\end{align*}
 where\,\footnote{All the constants $C_i$ below depend on $T$ and $\La$.} $C_2(T, \La)>0$   (see \eqref{2.3}).~This implies that $M$ maps  $ B_{C(J_{T_1};H^3_{(0)} )} (\hat \psi,1)$ into itself for sufficiently small    $T_1$ and $\de$. In a similar way, for~any~$\psi_1,\psi_2\in  B_{C(J_{T_1};H^3_{(0)} )} (\hat \psi,1)$ and $t\in J_{T_1}$, we have
 \begin{align*}
\|M(\psi_1)(t) - M(\psi_2)(t)\|_{(3)}\le &C_3\Big( 	( \|u\|_{L^2(J_{T_1};\R^q)} +T_1^{\frac12}) \|\psi_1 -  \psi_2 \|_{C(J_t;H^3_{(0)} ) } \Big).   
\end{align*}Thus, $M$ is a contraction in  $ B_{C(J_{T_1};H^3_{(0)} )} (\hat \psi,1)$ for sufficiently small~$T_1$ and $\de$. Hence, there is $\psi \in  B_{C(J_{T_1};H^3_{(0)} )} (\hat \psi,1)$ such that $M(\psi)=\psi$.

  {\it Step~3.~Existence up to $T$.}~Let $\psi$ be a maximal solution of \eqref{2.1},~\eqref{2.2}   
corresponding to $\psi_0\in H^3_{(0)}$ and~$u\in L^2(J_T,\R^q)$   satisfying~\eqref{2.4}. Then   there  is  a maximal time  $T_*\in J_T$ such that $\psi$ is defined on~$[0,T_*)$ and
$$
\|\psi(t)\|_{(3)}\to +\ty\quad \text{ as $t\to {T_*}^-$ when $T_*<T.$}
$$ The difference  $\varphi=\psi-\hat \psi$ satisfies the inequality 
\begin{align}
\|\varphi(t)\|_{(3)}& \le C_4\de+   
C_4\int_0^t \left(\|\varphi(s)\|_{(3)}+\|\varphi(s)\|_{(3)}^{2p+1}\right) \dd s\nonumber\\&\quad+C_4\left(\int_0^t |u(s)|^2 \|\varphi(s)\|_{(3)}^2  \dd s\right)^\frac12,\quad t\in [0,T_*).\label{3.8}
\end{align}  
  Let us denote 
$$
\tau=\inf\{t\in J_T: \|\varphi(t)\|_{(3)}=1\},
$$where the infimum over an empty set is equal to $T$.~Let us show that, for sufficiently small $\de$, we have $\tau=T$. Arguing by contradiction,  let us assume that~$\tau<T$. From \eqref{3.8} we derive  
$$
\|\varphi(t)\|_{(3)}^2 \le C_5\de^2+   
C_5\int_0^t \|\varphi(s)\|_{(3)}^2 (1+ |u(s)|^2 ) \dd s ,\quad t\in [0,\tau).
$$The Gronwall inequality implies  
$$
\|\varphi(t)\|_{(3)}^2 \le C_5\de^2 \exp\left(   
C_5\int_0^t   (1+ |u(s)|^2 ) \dd s \right)<1,\quad t\in [0,\tau)
$$for   small $\de$ and any $u\in L^2(J_T,\R^q)$   satisfying~\eqref{2.4}.~This contradicts the definition of $\tau$.
Thus   $\tau=T$ for small $\de$.
 
    {\it Step~4.~Differentiability.} The proof of $C^1$ regularity of the resolving operator is similar to the case   considered  in  Sections~2.2 and~3.2  in \cite{BL-2010}, so we do not provide the details.

\subsection{Proof of Lemma \ref{L:1.4}}\label{S:3.3}

Let us fix any $k\ge1$. Let $\phi_{k,\kappa}$ be the eigenfunction of the operator $A_\kappa$ (see~\eqref{1.7}) associated with the eigenvalue~$\la_{k,\kappa}$. By Theorem~3 in Chapter~2 in~\cite{PT-87}, both~$\phi_{k,\kappa}$ and $\la_{k,\kappa}$ are real-analytic functions in~$\kappa$.~By differentiating in $\kappa$ the    identity 
$$
\left(-\p_{xx}^2  -\pi^2+2 \kappa\phi_1^2 -\la_{k,\kappa}    \right)\phi_{k,\kappa}=0,
$$
  we obtain
$$
\left(-\p_{xx}^2  -\pi^2+2 \kappa\phi_1^2 -\la_{k,\kappa}    \right)\frac{\dd\phi_{k,\kappa}}{\dd \kappa}+\left(2 \phi_1^2 -\frac{\dd \la_{k,\kappa}}{\dd \kappa}    \right)\phi_{k,\kappa} =0.
$$Taking the scalar product in $L^2(I;\R)$ of this identity with  $\phi_{k,\kappa}$, we get
$$
\frac{\dd \la_{k,\kappa}}{\dd \kappa}  =\lag 2\phi_1^2,\phi_{k,\kappa}^2\rag_{L^2}> 0 
$$for any   $\kappa\in \R$. We conclude
  that $\la_{k,\kappa}$ is strictly increasing in $\kappa$, so it can vanish for at most one value of $\kappa\in \R$. Moreover,  as~$\la_{1,0}=0$, strict monotonicity of~$\la_{k,\kappa}$ implies that $0<\la_{1,\kappa}\le\la_{k,\kappa} $ for any $k\ge1$ and $\kappa>0$. This~completes the proof of the lemma.

\subsection{Saturation property}\label{S:3.4}
This section is devoted to  the proof of Proposition~\ref{P:1.5}. 
 It suffices to consider the case $q=2$ and $Q(x)=(\mek,\cos(\pi x))$. It is easy to see that 
 \begin{equation}\label{3.9}
 	\HH=\lspan_\R\{\phi_1, \phi_2\}.
 \end{equation}
  Let us denote 
	$$
	F(g)=i(-\p_{xx}^2 g -\pi^2g+2 \kappa  \phi_1^2 \Re(g)) \quad\textup{ for   } g\in H^2(I;\C).
	$$
From \eqref{1.3} and \eqref{1.4} it follows that    $F(g) \in \HH_j$ for   $g\in \HH_{j-1}$ and $j\ge1$. 
In~particular, $F(g)\in \HH_\ty$ for   $g\in \HH_\ty$. Furthermore, from \eqref{3.9}  we derive   
 $$
\HH_\ty  \subset \lspan_\C\{\phi_k:k\ge1\} \subset H^3_{(0)}.
$$ The proposition will be established  if we show that   $ \HH_\ty$ is dense in~$ H^3_{(0)}$. The~proof of this is divided into three steps.

{\it Step~1.}~First let us show that the eigenfunctions~$\phi_{2k+1}$, $k\ge0$ belong to~$\HH_\ty$.~We proceed by recurrence.~By \eqref{3.9}, we~have   $\phi_1\in\HH\subset  \HH_\ty$.~Let us take any $N\ge1$,  assume that   $\phi_{2k+1}\in\HH_\ty$ for any~$0\le k\le N-1$, and prove that $\phi_{2N+1}\in\HH_\ty$.   We have
$$
F(\phi_{2N-1})=i c_{2N-1} \phi_{2N-1}    +i2 \kappa  \phi_1^2 \phi_{2N-1} \in\HH_\ty,
$$ where $c_{2N-1}=\pi^2\left((2N-1)^2-1\right)$. Combining this  with the identity
$$
\phi_1^2 \phi_{2N-1}=\frac12\left( 2 \phi_{2N-1}-\phi_{2N-3} -\phi_{2N+1} \right),
$$
we obtain
$$
  F(F(\phi_{2N-1}))= \left(c_{2N-1}^2-c_{2N-1} \right) \phi_{2N-1}     +\frac12\left( c_{2N-3}\phi_{2N-3} +c_{2N+1}\phi_{2N+1} \right)\in\HH_\ty.
$$
  As $\phi_{2N-3}, \phi_{2N-1}\in \HH_\ty$,  we conclude that   $\phi_{2N+1}\in\HH_\ty$.  

In a similar way, as $\phi_2\in\HH\subset  \HH_\ty$,
one   shows that $\phi_{2k}\in \HH_\ty$ for any $k\ge1$.

 {\it Step~2.} Let us prove that the vector space  
 $$
 \GG=\lspan_{\R}\{i F(\phi_k):k\ge1\}
 $$ is dense in  $H^3_{(0)}(I;\R)$.~Indeed, let~$A_\kappa$ be the operator   defined by \eqref{1.7} with $\kappa \in \R\setminus\K$. Then $\la_{k,\kappa}\neq 0$ for any $k\ge1$, so  the image of $A_\kappa$ is dense in~$H^3_{(0)}(I;\R)$. It~remains to note that $A_\kappa g =-iF(g)$ for    $g\in H^2$, so $\GG$ is dense in~$H^3_{(0)}(I;\R)$.

 {\it Step~3.} Combining the results of steps 1 and 2, and using the fact that
 $$
 \lspan_{\R}\{  \phi_k, F(\phi_k):k\ge1\} \subset \HH_\ty,
 $$we see that  $ \HH_\ty$ is dense in~$ H^3_{(0)}$.~This  completes the proof of the proposition.

 \begin{corollary}\label{C:3.4}
Let  $ \K\subset (-\ty,0]$ be  the set in Lemma~\ref{L:1.4},   $Q(x)=(\mek,\cos(\pi x))$,  $V(x) = 0$, and   $W(x) =  2 \kappa  \phi_1^2(x)$.~Then, for any $\kappa\in \K$,  the codimension of  $ \HH_\ty$ in  $H^3_{(0)}$ is one.
 \end{corollary}
\begin{proof}
	As $\kappa\in \K$,   there is a unique   $k\ge1$ such that $\la_{k,\kappa}=0$.~From the   proof of Proposition~\ref{P:1.5} it follows that the vector space spanned by $i \phi_{k,\kappa}$ is the orthogonal complement of  $ \HH_\ty$ in  $H^3_{(0)}$. 
\end{proof}

\subsection{A closed image theorem}\label{S:3.5}

In this section, we formulate a simple functional analysis result    used in the proof of exact controllability of the linear Schr\"odinger~equation (see Section~\ref{S:1.2}). 
Being unable to find a  proper reference, we give a complete proof.

Let $X$ and $Y$   be Banach spaces, and let  $X^*$ and $Y^*$   be their duals. 

\begin{lemma}\label{L:3.3}
	  Assume that  $A: X\to Y$ and $B: X\to Y$ are linear continuous
operators such that $A$ is surjective and~$B$  is compact. Then the image of $A+B$ is closed in $Y$.
\end{lemma}
\begin{proof}
 Let $A^*:Y^*\to X^*$ and $B^*:Y^*\to X^*$ be the adjoint operators of $A$ and~$B$. By Theorem~2.19 in~\cite{HB-11}, the set  
  $(A+B)(X)$ is closed in $Y$ if and only if~$(A^*+B^*)(Y^*)$ is closed in $X^*$. We will prove that   $(A^*+B^*)(Y^*)$ is closed in~$X^*$ in  three steps.

 {\it Step~1.} Let us first show that the   kernel of the operator $A^*+B^*$ is finite-dimensional. To~this end, we prove  that any   bounded sequence $\{y_n\}$ in the kernel of $A^*+B^*$ has a convergent subsequence. Indeed, by Theorem~6.4 in~\cite{HB-11}, $B^*$~is compact. So there is a subsequence $\{y_{n_k}\}$ such that $\{B^*(y_{n_k})\}$ converges. The~equality~$A^*(y_{n_k})+B^*(y_{n_k})=0$ implies that $\{A^*(y_{n_k})\}$ also converges. As~$A^*$ is injective and $A^*(Y^*)$ is closed, the open mapping theorem implies  that~$\{y_{n_k}\}$ converges. 
 
 Thus the   kernel of $A^*+B^*$ is finite-dimensional, so it is complemented.  Without loss of generality, we can assume that $A^*+B^*$ is injective.

	{\it Step~2.} Let  $\{y_{n}\} \subset Y^*$ be a sequence such that $\{(A^*+B^*)(y_n)\}$ converges.  Let us show that~$\{y_{n}\}$ is bounded.~Arguing by contradiction, let us
	assume that there is a subsequence   such that $\|y_{n_k}\|_{Y^*}\to +\ty$. Then, for the sequence $\tilde y_n= y_{n_k}/\|y_{n_k}\|_{Y^*}$ we have $(A^*+B^*)(\tilde y_n)\to 0$ as $n\to +\ty$.~By the fact that~$\|\tilde y_n\|_{Y^*}=1$ and the compactness of~$B^*$, there is a subsequence $\{\tilde y_{\tilde n_k}\}$ such that~$\{B^*(\tilde y_{\tilde n_k})\}$ converges. From the equality
	$$
	A^*(\tilde y_{\tilde n_k})=(A^*+B^*)(\tilde y_{\tilde n_k})-B^*(\tilde y_{\tilde n_k})
	$$it follows that  $\{A^*(\tilde y_{\tilde n_k})\}$ converges too. As in step~1, this implies that $\{ \tilde y_{\tilde n_k}\}$ converges to some limit $\tilde y$.
	 Since~$\|\tilde y_{\tilde n_k}\|_{Y^*}=1$, we have   $\|\tilde y\|_{Y^*}=1$.	 By~continuity, 
	\begin{equation}\label{3.10}
		(A^*+B^*)(\tilde y_{\tilde n_k})\to (A^*+B^*)(\tilde y) \quad\textup{as } k\to +\ty. 	
	\end{equation}
	On the other hand, from the construction of $\{\tilde y_n\}$ it follows that 
	\begin{equation}\label{3.11}
(A^*+B^*)(\tilde y_{\tilde n})\to 0 \quad\textup{as } n\to +\ty.	
	\end{equation}
From \eqref{3.10} and \eqref{3.11} it follows that
  $\tilde y$ is a non-zero element of the kernel of~$A^*+B^*$. This  contradicts the injectivity of $A^*+B^*$.
	
		{\it Step~3.} Let  $\{y_{n}\} \subset Y^*$ be a bounded sequence such that  
		  \begin{equation}\label{3.12}
		     (A^*+B^*)(y_n)\to x\quad\textup{as } n\to +\ty.
		  \end{equation}By   compactness of $B^*$, there is a subsequence $\{y_{n_k}\}$ such that $\{B^*(y_{n_k})\}$ converges. From \eqref{3.12} it  follows that~$\{A^*(y_{n_k})\}$ converges too.~As above, this implies that $\{y_{n_k}\}$ converges to some limit $y$.~By continuity and \eqref{3.12}, we have~$ (A^*+B^*)(y)=x$.~We conclude   that the set   $(A^*+B^*)(Y^*)$ is closed in~$X^*$.
		    \end{proof}

  \addcontentsline{toc}{section}{Bibliography}
\def\cprime{$'$} \def\cprime{$'$}
  \def\polhk#1{\setbox0=\hbox{#1}{\ooalign{\hidewidth
  \lower1.5ex\hbox{`}\hidewidth\crcr\unhbox0}}}
  \def\polhk#1{\setbox0=\hbox{#1}{\ooalign{\hidewidth
  \lower1.5ex\hbox{`}\hidewidth\crcr\unhbox0}}}
  \def\polhk#1{\setbox0=\hbox{#1}{\ooalign{\hidewidth
  \lower1.5ex\hbox{`}\hidewidth\crcr\unhbox0}}} \def\cprime{$'$}
  \def\polhk#1{\setbox0=\hbox{#1}{\ooalign{\hidewidth
  \lower1.5ex\hbox{`}\hidewidth\crcr\unhbox0}}} \def\cprime{$'$}
  \def\cprime{$'$} \def\cprime{$'$} \def\cprime{$'$}

\end{document}